\documentclass[11pt]{article}

\usepackage{epsfig}
\usepackage{amssymb, latexsym}
\usepackage{amscd}
\usepackage[all]{xy}
\usepackage{epsf}
\usepackage[mathscr]{eucal}

\usepackage{tikz}
\usepackage{tikz-cd}
\usepackage{verbatim}
\usetikzlibrary{matrix,arrows}
\usepackage{mathtools}
\usepackage{xcolor}
\usepackage{stmaryrd}
\usetikzlibrary{positioning,fit,calc}

\DeclareFontFamily{U}{solomos}{}
\DeclareErrorFont{U}{solomos}{m}{n}{10}
\DeclareFontShape{U}{solomos}{m}{n}{
  <-> s*[1.1]  gsolomos8r
}{}

\usepackage{amsmath,amsfonts,amscd,amssymb,amsthm}

\usepackage[titletoc]{appendix}
\usepackage[sc]{titlesec}

\long\def\comment#1\endcomment{}

\theoremstyle{plain}

\newtheorem{theorem}{\sc Theorem}[section]
\newtheorem{lemma}[theorem]{\sc Lemma}
\newtheorem{conjecture}[theorem]{\sc Conjecture}
\newtheorem{prop}[theorem]{\sc Proposition}
\newtheorem{coroll}[theorem]{\sc Corollary}

\makeatletter
\newcommand*{\doublerightarrow}[2]{\mathrel{
  \settowidth{\@tempdima}{$\scriptstyle#1$}
  \settowidth{\@tempdimb}{$\scriptstyle#2$}
  \ifdim\@tempdimb>\@tempdima \@tempdima=\@tempdimb\fi
  \mathop{\vcenter{
    \offinterlineskip\ialign{\hbox to\dimexpr\@tempdima+1em{##}\cr
    \rightarrowfill\cr\noalign{\kern.5ex}
    \rightarrowfill\cr}}}\limits^{\!#1}_{\!#2}}}
\makeatother

\theoremstyle{plain}

\theoremstyle{exercise}
\newtheorem{remark}[theorem]{\sc Remark}

\makeatletter \@addtoreset{equation}{section} \makeatother

\def\eqref#1{\thetag{\ref{#1}}}

\let\latexref=\ref
\def\ref#1{{\normalfont{\latexref{#1}}}}

\newcommand{\ldot}{{\:\raisebox{2,3pt}{\text{\circle*{1.5}}}}}
\newcommand{\udot}{{\:\raisebox{3pt}{\text{\circle*{1.5}}}}}
\def\dlim_#1{{\displaystyle\lim_{#1}}^\hdot}

\newcommand{\cchar}{\operatorname{\sf char}}

\newcommand{\Ker}{\operatorname{{\rm Ker}}}

\newcommand{\id}{{\mathrm{id}}}

\newcommand{\Mor}{\mathrm{Mor}}
\newcommand{\Ob}{\mathrm{Ob}}

\newcommand{\Hom}{\mathrm{Hom}}

\newcommand{\Hoch}{\mathrm{Hoch}}

\newcommand{\op}{\mathrm{op}}
\newcommand{\fin}{\mathrm{fin}}
\newcommand{\Mon}{{\mathscr{M}{on}}}
\newcommand{\dg}{\mathrm{dg}}
\newcommand{\Ho}{\mathrm{Ho}}

\newcommand{\Sets}{\mathscr{S}ets}

\newcommand{\Alg}{{\mathrm{Alg}}}

\newcommand{\Cat}{{\mathscr{C}at}}
\newcommand{\Top}{{\mathscr{T}op}}

\renewcommand{\k}{\Bbbk}

\newcommand{\Assoc}{\mathbf{Assoc}}

\renewcommand{\cchar}{\mathrm{char}\ }

\newcommand{\coh}{\mathrm{coh}}

\newcommand{\sotimes}{{\overset{\sim}{\otimes}}}

\newcommand{\Sym}{\mathrm{Sym}}

\newcommand{\Br}{\mathbf{Br}}

\newcommand{\SC}{\mathrm{SC}}
\newcommand{\scc}{\mathbf{SC}}

\newcommand{\e}{\mathbf{E}}
\newcommand{\Eq}{\mathrm{Eq}}
\newcommand{\Act}{\mathrm{Act}}
\newcommand{\comm}{\mathrm{comm}}
\newcommand{\sO}{{\mathbf{s}\O}}
\newcommand{\norm}{\mathrm{norm}}

\newcommand{\sevafigc}[4]{\begin{figure}[h]\centerline{
 \epsfig{file=#1,width=#2,angle=#3}}
\bigskip\caption{#4}\end{figure}}

\title{\sc{Twisted tensor product of dg categories and Kontsevich's Swiss Cheese conjecture}}

\author{\sc{Michael Batanin and Boris Shoikhet}}
\date{}

\begin{document}\maketitle
{\footnotesize
\begin{center}{\parbox{4,5in}{{\sc Abstract.}
Let $A$ be a $1$-algebra. The Kontsevich Swiss Cheese conjecture [K2] states that the homotopy category $\Ho(\Act(A))$ of actions of  $2$-algebras on $A$ has a final object and that this object is weakly equivalent to the pair $(\Hoch^\udot(A), A),$ where $\Hoch^\udot(A)$ is the Hochschild complex of $A$. Here the category of actions is the category whose objects are pairs $(B,A)$ which are algebras of the chain Swiss Cheese operad such that the induced action of the little interval operad on the component $A$ coincides with the $1$-structure on $A$.

We prove that there is a colored dg operad $\mathcal{O}$ with 2 colors, weakly equivalent to the chain Swiss Cheese operad  
for which the following ``stricter" version of the Kontsevich Swiss Cheese conjecture holds.  
Denote the two colors of $\mathcal{O}$ by $a$ (for the 1-algebra argument) and $b$ (for the 2-algebra argument), denote by $E_1^{\mathcal{O}}$ the restriction of $\mathcal{O}$ to the color $a$, and by $E_2^{\mathcal{O}}$ the restriction of $\mathcal{O}$ to the color $b$. Let $\Alg(\mathcal{O})$ be the category of dg algebras over $\mathcal{O}$.
For a fixed $1$-algebra $A$ we also have the the category of action $\Alg(\mathcal{O})_A$ (equal to $\Act(A)$ in case of the Swiss Cheese operad). We prove that there is an equivalence of categories
$$\Alg(\mathcal{O})_A \cong \Alg(E_2^{\mathcal{O}})/\Hoch^\udot(A)$$ 
We stress that for this particular model of Swiss Cheese operad the statement holds on the chain level, {\it without passage to the homotopy category}.

}}
\end{center}
}

\section{\sc Introduction}
Throughout the paper, $\k$ is a field of any characteristic. The category of complexes of $\k$-vector spaces is denoted by $C^\udot(\k)$. For a colored operad $\mathcal{O}$ with colors $a,b,c,\dots,z$ we denote by $\mathcal{O}((a^{m_1}, b^{m_2},c^{m_3},\dots, e^{m_k}; z)$ the components of the arity $m_1+\dots+m_k$ whose first $m_1$ arguments have color $a$, second $m_2$ arguments have color $b$, and so on, and the output color is $z$. Recall that the latter object is endowed with an action of the product of symmetric groups $\Sigma_{m_1}\times\dots\times \Sigma_{m_k}$. 

\subsection{\sc }
We denote by $E_n$ the topological little disks operad in dimension n. 
Let $A,B\in \Top$ such that $B$ is an $E_2$-algebra and $A$ is an $E_1$-algebra. The topological Swiss Cheese operad (introduced by A. Voronov in [V]) makes it possible to define an {\it action} of $B$ on $A$ (in the sense of these structures). 

The topological operad $\SC_1$ is defined as a colored symmetric operad in topological spaces with 2 colors $a$ and $b$, with the allowed operations of type $\SC_1(b^n;b)=E_2(n)$, $n\ge 0$ (this suboperad is thus the little disks operad $E_2$), and the operations of type $\SC_1(b^n, a^m; a)$, $m, n\ge 0$ (which in particular case $\SC_1(b^0, a^m,a)=E_1(m)$ gives the operad $E_1$). The topological space $\SC_1(b^n,a^m,a)$ is the configuration space of non-intersecting disks and half-discs inside a unit half-disc, see Figure \ref{fig1}. The disks are labelled from 1 to $n$, and the half-disks are labelled from 1 to $m$, the symmetric groups acts by changing the labelling. We assume that $\SC_1(b^0,a^0;a)=*$, $\SC_1(b^0,a^1;a)=*$ (the radius 1 upper half-disk equal to the outer half-disc is the only allowed configuration). Note that $\SC_1(b^0,a^0;a)$ is a contractible topological space. The compositions in the colored operad $\SC_1$ are clear.
Such an operad has a clear counterpart in any dimension. It was introduced and studied by A.Voronov [V].

\sevafigc{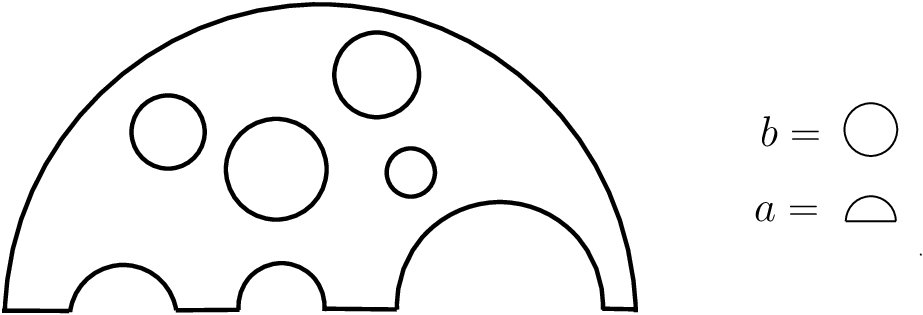}{80mm}{0}{The Swiss Cheese operad component $\SC_1(b^n, a^m; a)$. \label{fig1}}

One has natural operad embeddings $i\colon E_1\to \SC_1$ and $j\colon E_2\to\SC_1$ which send the only color of the source operad to the color $a$ and the color $b$, correspondingly. 

An algebra over $\SC_1$ is a pair $A,B$ of topological spaces (so that $A$ corresponds to the color $a$ and $B$ corresponds to the color $b$) equipped with maps 
$$
\SC_1(b^n;b)\times B^{\times n}\to B,\ \ \SC_1(b^n,a^m;a)\times B^{\times n}\times A^{\times m}\to A
$$
which are subject to the natural compatibility with the operadic compositions, and the equivariance with respect to the symmetric group actions. A structure of an algebra $(B,A)$ over $\SC_1$ results to an $E_2$-algebra structure on $B$, an $E_1$-algebra structure on $A$, {\it and} some action of $B$  on $A$. The latter action is what the Swiss Cheese operad is used for. A dimension 0 counterpart is untuitively more clear: an action of a monoid $M$ on a set $X$. The Swiss Cheese operad provides a concept of such action in any dimension, so that one can speak on an $E_{n+1}$-algebra acting on $E_n$-algebra. 

In what follows we consider the chain versions of the topological operads. The functor of normalised cubical chains from $\Top$ to $C^\udot(\k)$ has a natural lax-monoidal structure. Therefore, for any topological operad $\mathcal{O}$ the normalised cubical chains $C_\ldot(\mathcal{O}(=);\k)$ forms an operad in $C^\udot(\k)$. 
We use the notation $\e_n=C^\udot(E_n;\k)$.

\subsection{\sc}
Denote by $\scc_1$ the chain version of the Swiss Cheese operad in dimension 1.
We denote by $(B,A)$ an algebra over the chain operad $\scc_1$. It means that $B,A\in C^\udot(\k)$, and there are maps
$$
\scc_1(b^n;b)\otimes B^{\otimes n}\to B,\ \ \scc_1(b^n,a^m;a)\otimes B^{\otimes n}\otimes A^{\otimes m}\to A
$$
which are subject to the natural identities. 
%(In what follows we often abbreviate $\scc_1$ as $\scc$, as we deal with dimension 1 only in this paper). 

In particular, $B$ becomes a $\e_2$-algebra, $A$ a $\e_1$-algebra, and the fact that $(B,A)$ such above form an algebra over the dg operad $\scc_1$ can be informally interpreted that the $\e_1$-algebra $A$ is a ``module'' over a $\e_2$-algebra $B$. 

The Deligne conjecture (which by now admits several proofs [MS], [KS], [T3], also [B1] and [T1] together) says that the cohomological Hochschild complex $\Hoch^\udot(A)$ of a $\e_1$-algebra $A$ has a $\e_2$-algebra structure. 

The Kontsevich Swiss Cheese conjecture [K2] (see below) implies that there is a canonical (up to a homotopy) solution to the Deligne conjecture, for any $A$. It conjectured that the pair $(\Hoch^\udot(A),A)$ is an algebra over $\scc_1$, and gives a universal property for it. The fact the the pair $(\Hoch^\udot(A),A)$ is a $\scc_1$-algebra was proven in [DTT] some 10 years later. (The proof is essentially non-trivial and is relied on the theory of colored 2-operads [B1] and some results of loc.cit. on the derived symmetrisation of contractible colored 2-operad of Swiss Cheese type). 

Fix an associative algebra $A$ (which we consider as an $\e_1$-algebra under the canonical quasi-isomorphism $p\colon \e_1\to\Assoc$). 
Define a category $\Act(A)$ as follows.

An object of $\Act(A)$ is an $\e_2$-algebra $B$ such that $(B,A)$ form an algebra over the chain operad $\scc_1$. A morphism $(B,A)\to (B^\prime,A)$ is a morphism $\phi\colon B\to B^\prime$ of $\e_2$-algebras such that $(\id_A,\phi)\colon (B,A)\to (B^\prime,A)$ is a map of algebras over $\scc_1$. Such a morphism is called a quasi-isomorphism if $\phi$ is a quasi-isomorphism of complexes. The (Gabriel-Zisman) localisation of the category $\Act(A)$ by quasi-isomorphisms is denoted by $\Ho(\Act(A))$. We call it the homotopy category of $\Act(A)$.

The {\it Kontsevich Swiss Cheese conjecture} is the following statement:

For any associative dg algebra $A$, the homotopy category $\Ho(\Act(A))$ has a final object $(B_\fin,A)$. Moreover, $B_\fin$ is quasi-isomorphic, as a complex, to $\Hoch^\udot(A)$.

Thus, the conjecture suggested that $(\Hoch^\udot(A),A)$ is an algebra over $\scc_1$ (which itself is a stronger statement than the Deligne conjecture, it was proven in [DTT]), and that for any other algebra $(B,A)$ over $\scc_1$ there is a unique, up to a weak equivalence, map of $\e_2$-algebras $\phi\colon B\to \Hoch^\udot(A)$ such that the $B$-module structure on $A$ is the restriction along $\phi$ of the $\Hoch^\udot(A)$-module structure on $A$. Note that a solution to this conjecture provides a {\it canonical} up to a weak equivalence solution to the Deligne conjecture.

A version of the Kontsevich Swiss Cheese conjecture (slightly weaker then the original one)  is proven in [Th] for topological operads, for any dimension.  The proof is very involved, though, and does not seem to be applicable directly to the chain operads  which motivated us for finding a more direct and conceptual approach.  

\comment
Our main result in this paper proves yet another version of the Kontsevich conjecture for dg-enriched case which is still  slightly weaker than the original one, but does provide a unique up to a weak equivalence solution to the Deligne conjecture (in dimension $2$) and still characterises Hochschild complex by a universal property (which is the main result of [Th] in topological case). Also, our approach works for any operad weakly equivalent to $\scc_1$.  

Let us stress the following general technical difficulty when dealing with the homotopy category of $\Act(A)$. 

Recall (see Section \ref{section4} for more detail) that under some rather mild assumptions, the category of dg algebras over a (colored) dg operad $\mathcal{O}$ has a projective model (or, at least, semimodel) structure. For $\cchar \k=0$ it is always the case, and for general $\k$ one has to require that the operad in $\Sigma$-cofibrant that is, the products of the symmetric groups act freely on the (colored) arity components (see [Hi],[BM1,2],[BB] for references). 

Consequently, when $A$ is {\it variable}, the category $\Alg(\scc_1)$ %whose objects are pairs $(B,A)$ as above, and whose morphisms are pairs $(\phi,f)$ where $\phi\colon B\to B^\prime$ and $f\colon A\to A^\prime$ such that $(\phi,f)$ of algebras over $\scc_1$ 
has a projective model structure. However, 
the category $\Act(A)$ is {\it not} the category of algebras over an operad, and it does not inherit a closed model structure in any natural way.  One can treat it as the {\it set-theoretical fibre} 
$p^{-1}(A)$ for the forgetful functor $p\colon \Alg(\scc_1)\to \Alg(\e_1)$.  This functor is not a Grothendieck fibration or opfibration and one can not expect, in general, an existence of the induced  model structure on this set-theoretical fibre. We even can not guarantee that this fiber has a terminal object (which is a necessary requirement for the existence of model structure).  On the other hand, a model structure exists for the {\it homotopy fibre}: for a closed model category $\mathcal{M}$ and $X\in \mathcal{M}$ an object, the comma-categories $\mathcal{M}\backslash X$ and $X\backslash\mathcal{M}$ are naturally closed model categories, in which a morphism is a weak equivalence (resp., a cofibration, resp., a fibration) iff the corresponding morphism in $\mathcal{M}$ is [Hir] Sect. 7.6.4, [Hov] Sect. 1.1. A natural strategy to prove the original form of the Swiss-Cheese conjecture would be to establish a Quillen equivalence between two model categories $\Act(A)$ and $\Alg(\e_2)/\Hoch^\udot(A)$ but the  absence of an induced model structure on the set-theoretical fibre $\Act(A)$ constitutes a difficulty in solving the Swiss Cheese conjecture. 
\endcomment

\comment
\subsection{}
Here we state the main result of this paper. 

We replace the category $\Ho(\Act(A))$ by a category more flexible from the homotopy theory point of view. Call a morphism $(\phi,f)\colon (B,A)\to (B^\prime, A^\prime)$ in $\Alg(\scc)$ a quasi-isomorphism if both $\phi$ and $f$ are quasi-isomorphisms of complexes. The localisation of $\Alg(\scc)$ by quasi-isomorphisms is denoted by $\Ho(\Alg(\scc_1))$. One has a functor $p_\Ho\colon \Ho(\Alg(\scc))\to\Ho(\Alg(\e_1))$. Take any associative dg algebra $A$ and consider the set-theoretical fibre \\ $\Ho(\Alg(\scc))_A:=p_\Ho^{-1}(A)$.

\begin{theorem}\label{theoremintro}
For any associative dg algebra $A$ over a field $\k$ of any characteristic the category $\Ho(\Alg(\scc))_A$ has a final object. The  $\e_2$-part of it is weakly equivalent to the Hochschild complex $\Hoch^\udot(A)$ as a complex, and thus provides a canonical up to a weak equivalence $\e_2$-algebra structure on it. The same statement holds for any $\Sigma$-cofibrant colored dg operad with 2 colors, weakly equivalent to $\scc_1$. In case $\cchar\k=0$ the condition of $\Sigma$-cofibrancy of the operad can be dropped. 
\end{theorem}
\endcomment

%Recall (see Section 4) that under the same assumptions which we made for existence of transferred closed model structure on the category of dg algebras over an operad (namely, no assumptions for the case $\cchar \k=0$, and $\Sigma$-cofibrancy of the operad for general $\k$) it is true that a quasi-isomorphism $f\colon \mathcal{O}_1\to \mathcal{O}_2$ defines a Quillen equivalence between the model categories $\Alg(\mathcal{O}_1)\rightleftarrows \Alg(\mathcal{O}_2)$, with right adjoint the restriction. 

\comment
Recall that a (colored) operad in a monoidal model category $\mathcal{M}$ is called $\Sigma$-cofibrant [BM1], Sect. 4, if the (product of the) symmetric groups act freely on its arity components. If the monoidal model category admits a {\it commutative Hopf interval} [BM1], Sect. 3, the result of [BM2] Th. 2.1 says the category of algebras over a $\Sigma$-cofibrant (colored) operad has a closed model structure, and if $\mathcal{M}$ is moreover left proper, a weak equivalence of colored $\Sigma$-cofibrant operads leads to a Quillen equivalence of the categories of algebras over them [BM2], Th.4.1. For $\mathcal{M}=C^\udot(\k)$ and $\cchar \k=0$, the same is true with $\Sigma$-cofibrancy condition dropped [Hi], Th. 4.1.1.
\endcomment

\subsection{\sc}
Within our approach, proposed in this paper, a proof of the Swiss Cheese conjecture is divided into two steps. As the first step, we find a particular perfect model $\mathcal{O}_0$ of the chain Swiss Cheese $\scc_1$ for which the Swiss Cheese conjecture holds before passing to the localisation by quasi-isomorphisms. (It is an interesting question to which extent existence of such a perfect model is a general phenomenon, see Concluding remarks at the end of this Introduction for a further discussion). 

Our main result is:
\begin{theorem}\label{theorem1}
Let $\k$ be a field of any characteristic, $A$ a dg algebra over $\k$. There is a colored operad with 2 colors in $C^\udot(\k)$ weakly equivalent to the chain Swiss Cheese operad $\scc_1$ for which the following statement holds. Consider the category $\Alg(\mathcal{O}_0)_A$ of pairs $(B,A)\in \Alg(\mathcal{O}_0)$ whose second component is $A$, the morphisms are identity of $A$ when restricted to the second component. Then there is an equivalence of categories 
$$
\Alg(\mathcal{O}_0)_A\simeq \Alg(E_2^{\mathcal{O}_0})/\Hoch^\udot(A)
$$
In particular, the homotopy category $\Ho(\Alg(\mathcal{O}_0)_A)$ has a final object, given by $(\Hoch^\udot(A),A)$.
\end{theorem}
% In a sequel paper, we show how a statement about homotopy categories  of action  very close to the original Swiss Cheese conjecture  can be deduced for other weakly equivalent models of $\scc_1$ from
%Theorem \ref{theorem1}.
 
In fact, this operad $\mathcal{O}_0$ is well-known: it is the {\it reduced} $\k$-linear condensation of the simplicial version of Swiss Cheese operad of natural operations introduced and studied in [DTT]. However, the fact that $\Alg(\mathcal{O}_0)_A$ has a final object is not clear from the description of $\mathcal{O}_0$ in loc.cit., this statement is new. 

The second step, which we consider in our next paper, is a deduction of Swiss Cheese like statement for any $\Sigma$-cofibrant reduced chain operad weakly equivalent to $\scc_1$ from the statement of Theorem\ref{theorem1}. Here the idea is to provide a homotopy-theoretical version of the statement of Theorem \ref{theorem1}, making it possible to replace $\mathcal{O}_0$ by another weakly equivalent $\Sigma$-cofibrant operad in $C^\udot(\k)$. Note that the fact that $\mathcal{O}_0$ and $\scc_1$ are weakly equivalent is proven in [DTT, Th. 2.1]. 

\subsection{\sc }
We give an alternative description of the same colored dg operad $\mathcal{O}_0$, relied on the skew-monoidal structure $\sotimes$ on the category $\Cat_\dg(\k)$ of small dg categories over $\k$ recently introduced in [Sh1]. (This skew-monoidal structure was called {\it twisted tensor product} in loc.cit., which term is a bit confusing; a better name would be {\it Gray-like tensor product}). We outline its construction and properties in Section \ref{sectiontw}. 

Then an algebra over $\mathcal{O}_0$ can be interpreted as a pair $(B,A)$, where $A,B$ are dg algebras over $\k$, $B$ is moreover a {\it monoid} in $(\Cat_\dg(\k), \sotimes)$ acting on $A\in \Cat_\dg(\k)$. It means precisely the same as an action of an ordinary monoid in $\Sets$ on a set, the only difference is that the monoidal-like structure defined by $\sotimes$ is only skew-monoidal [S] (meaning existence of a one-side associator, whose map is neither an isomorphism not a quasi-equivalence, satisfying the usual pentagon and unit axioms). So it means that there are maps of dg algebras
\begin{equation}\label{skewmonoid}
m\colon B\sotimes B\to B,\ \ \ell\colon B\sotimes A\to A
\end{equation}
such that 
\begin{itemize}
\item[(i)] $m$ is associative, in the sense that the diagram
\begin{equation}\label{skewmonoid1}
\xymatrix{
(B\sotimes B)\sotimes B\ar[d]_{m\sotimes\id}\ar[rr]^{\alpha}&&B\sotimes(B\sotimes B)\ar[d]^{\id\sotimes m}\\
B\sotimes B\ar[r]^m&B&B\sotimes B\ar[l]_m
}
\end{equation}
commutes,
\item[(ii)] $\ell$ defines a module structure, in the sense that the diagram commutes:
\begin{equation}\label{skewmonoid2}
\xymatrix{
(B\sotimes B)\sotimes A\ar[rr]^{\alpha}\ar[d]_{m\sotimes\id}&&B\sotimes (B\sotimes A)\ar[d]^{\id\sotimes\ell}\\
B\sotimes A\ar[r]^\ell&A&B\sotimes A\ar[l]_\ell
}
\end{equation}
\item[(iii)] $m(1_B\sotimes b)=m(b\sotimes 1_B)=b$,   $\ell(1_B\sotimes a)=a$ for any $a\in A$, $b\in B$ (the unit maps for $\sotimes$ are implicitly used here).
\end{itemize}
We stress that (iii) means that the units for both monoid structures (with respect to two different monoidal structures) are the same. 

A crucial new observation is the following 
\begin{lemma}\label{lemmaintro1}
A dg algebra $B$ is a monoid with respect to $\sotimes$ sharing the dg algebra unit if and only if it is an algebra over the braces operad $\Br$ of Getzler-Jones, whose underlying differential on the tensor coalgebra $T^\vee_+(B[1])$ is given by the bar-differential of the underlying dg algebra $B$. 
\end{lemma}
We recall the operad $\Br$ in Section \ref{sectionbrace}.
The proof is given by a direct but nice computation, see Section \ref{prooflemma1}.

The brace operad is exactly the restriction to the color $c$ suboperad of $\mathcal{O}_0$, so it is not a surprise that the data of $(B,A)$ satisfying (i)-(iii) above is exactly the data of an algebra over the colored operad $\mathcal{O}_0$.

Thus the Swiss Cheese property for $\Alg(\mathcal{O}_0)_A$ can be proven in vein of the general argument of [B2], which is outlined in the next Subsection.

\subsection{\sc }
A very baby version of the Swiss Cheese conjecture is the following elementary statement:

{\it Let $S$ be a set. Then the category of pairs (a monoid $M$, an $M$-module structure on $S$) is equivalent to the category of monoid maps $M\to \Hom(S,S)$, that is, to the comma-category of monoids over the monoid $\Hom(S,S)$.}

In fact, the category of sets is closed monoidal,
\begin{equation}\label{adjbasic}
\Sets(C\times D,E)\simeq \Sets(C, [D,E])
\end{equation}
where $[D,E]$ is the set $\Sets(D,E)$, but, keeping in mind some development of this idea, we regard it as an internal $\Hom$. 

It follows that, for a monoid $M$,
\begin{equation}\label{adjmon}
\Sets(M\times S,S)\simeq \Sets(M,[S,S])
\end{equation}
In both sides there stand the maps of sets. On the other hand, we like to replace the r.h.s. $\Hom$ by maps of {\it monoids}
and at the l.h.s. consider only those maps which endow $S$ with an $M$-set structure. It goes automatically if we interpret both ``refined'' Homs as suitable equalisers, namely
\begin{equation}\label{equal1}
\{\text{$M$-set structures on $S$}\}=\Eq\big(\Sets(M\times S, S)\rightrightarrows \Sets(M\times M\times S,S)\big)
\end{equation}
and 
\begin{equation}\label{equal2}
\Hom_\Mon(M,[S,S])=\Eq\big(\Sets(M,[S,S]))\rightrightarrows\Sets(M\times M,[S,S])\big)
\end{equation}
and note that under the adjunction \eqref{adjbasic} one pair of arrows corresponds to another. 

In particular, for a given set $S$, the category whose objects are monoids $M$ acting on $S$, and whose morphisms are maps of monoids $M\to [S,S]$, has a final object: it is, in terms of the r.h.s. of the adjunction, the identity map $\id: M=[S,S]\xrightarrow{\id} [S,S]$. So it is the baby version of the Swiss Cheese type statements. 

Next, one can categorify the above scheme. It means that we consider a monoidal category acting on a category (instead of a monoid acting on a set). For categories there is an adjunction similar to \eqref{adjbasic}:

\begin{equation}\label{adjbasicbis}
\Cat(C\times D,E)\simeq \Cat(C,[D,E])
\end{equation}
where $\Cat(-,=)$ denotes the set of functors, and 
 $[D,E]$ is the category whose objects are functors $F\colon D\to E$, and whose morphisms from $F$ to $G$ is the set of natural transformations. 

Let $M$ be a monoidal category and $C$ a category. Then \eqref{adjbasicbis} gives 
\begin{equation}\label{adjmonbis}
\Cat(M\times C,C)\simeq \Cat(M,[C,C])
\end{equation}
Note that $[C,C]$ is a monoidal category with the monoidal product equal to the composition. 

In both sides of \eqref{adjmonbis} there are subsets with objects formed by strict actions of a monoidal category on a category and by strict monoidal functors of monoidal categories, correspondingly. They are defined via equalisers analogous to \eqref{equal1}, \eqref{equal2}, correspondingly, which correspond one to another by the adjunction.  In conclusion,
the set of strict actions of $M$ on $C$ is isomorphic to the set of strict monoidal functors $M\to [C,C]$. 

In particular, one gets the following Swiss Cheese type statement: for a given category $C$, the category $\Act(C)$ whose objects are strict monoidal categories $M$ strictly acting on $C$, and whose morphisms are strict monoidal functors, has a final object. It is the identity map $m=[C,C]\xrightarrow{\id}[C,C]$.

Now we can ``decategorify'' the previous statement, meaning restricting to the case of categories with a single object. 
By the Eckmann-Hilton argument, a monoidal category with a single object is just a commutative monoid $K$. Its single object should map to the unit of the monoidal category and should be mapped by a monoidal functor to the identity functor of $C$. Thus what we get in the r.h.s. is a monoid map
\begin{equation}
K\to [C,C](\id,\id)=Z(C)
\end{equation}
(where $Z(C)$ for the single object category $C$ is the center of the monoid $C$.

We conclude that, for a given monoid $C$, an action of a commutative monoid $K$ on $C$ is the same as the monoid map $K\to Z(C)$. In particular, the category $\Act^\comm(C)$ has a final object, given by $\id\colon Z(C)\to Z(C)$.

Clearly all these claims remain true for the $V$-enriched (monoidal) categories for any closed monoidal category $V$.

The idea is that the Kontsevich Swiss Cheese conjecture should be a suitably derived version of the latter statement on a commutative monoid action. Indeed, working with too restrictive strict context, we get collapses in both sides: in a relaxed context, the Eckmann-Hilton argument is not literally applied, and instead of a commutative monoid one should get roughly an $\e_2$-algebra, and similarly in the strict case $Z(C)$ for an algebra $C$ is just 0-th Hochschild cohomology (for dg enrichment), which should be replaced by the entire Hochschild complex in a suitable relaxed context. 

In the present paper we show that the twisted tensor product $\sotimes$ on $\Cat_\dg(\k)$ provides such a relaxed context.

The twisted tensor product of small dg categories is left adjoint to the ``coherent internal Hom'' defined as follows. Let $C,D$ be small dg categories. The category $[C,D]_\coh$ has dg functors $F\colon C\to D$ as objects, and for two dg functors $F,G\colon C\to D$ the Hom-complex is the reduced Hochschild complex $[C,D]_\coh(F,G)=\Hoch^\udot(C, M_{F,G})$, where $M_{F,G}$ is a $C$-bimodule defined as $M_{F,G}(X,Y)=D(FX,GY)$, $X,Y\in \Ob C$. We discuss motivations for this definition in Section \ref{sectioncohhom}, see Remark \ref{remcohhom}. Then the functor $[D,?]_\coh$ has a left adjoint $?\sotimes D$:
\begin{equation}\label{adjtw}
\Cat_\dg(C\sotimes D,E)\simeq \Cat_\dg(C, [D,E]_\coh)
\end{equation}
The dg category $C\sotimes D$ is called {\it the twisted tensor product} of dg categories $C$ and $D$. 

Due to this adjunction, the previously employed way of arguing is applied. Indeed, by Lemma \ref{lemmaintro1} a dg algebra is a $\sotimes$-monoid iff its a brace algebra (and the brace operad is known to be weakly equivalent to $\e_2$ [T4]). Also the internal Hom complexes $[C,D]_\coh(F,G)$ are given by the Hochschild complexes with some coefficients. In particular case, $[C,C](\id,\id)=\Hoch^\udot(C)$. 

Thus, working with $(\Cat_\dg,\sotimes)$ the Eckmann-Hilton argument breaks to the expected statement of Lemma \ref{lemmaintro1}, and the space $[C,C]_\coh(\id,\id)$ also becomes equal to the expected $\Hoch^\udot(C)$ (instead of its 0-th cohomology $Z(C)$).

It shows that $\Alg(\mathcal{O}_0)_A$ has a final object isomorphic to $(\Hoch^\udot(A),A)$ for any dg algebra $A$.

\comment

\subsection{\sc } 
One easily sees from the latter statement that $(\Hoch(A), A)$ remains a final object in the localised category $\Ho(\Alg(\mathcal{O}_0)_A)$. Unfortunately there are no homotopy theoretical means (at least we are not aware about such means) to deduce from this statement a similar statement for a quasi-isomorphic to $\mathcal{O}_0$ colored dg operad $\mathcal{O}$. Indeed, as we already observed the set-theoretical fibre does not inherit any model structure from the ambient model category in general, so the standard homotopical technique can not be applied. On the other hand, the original conjecture is stated for the  weakly equivalent to $\mathcal{O}_0$ colored dg operad $\scc$, so a way to pass from one operad to a quasi-isomorphic one is needed. 

Our approach relies on the following

\begin{prop}\label{propintro1}
For a dg algebra $A$, the set-theoretical fibre of the homotopy category $\Ho(\Alg(\mathcal{O}_0))_A$ has a final object. Moreover, this final object is still $(\Hoch^\udot(A),A)$ corresponded to the identity map of $\Hoch^\udot(A)$ by adjunction \eqref{adjtw}.
\end{prop}
The proof is not really straightforward. In particular, it essentially relies on results of B.Keller from his nice paper [Ke] on derived invariance of Hochschild complex.  We prove Proposition \ref{propintro1} in Section \ref{section3}.

\comment
Let $\mathcal{O}$ be a $\Sigma$-cofibrant colored operad in a symmetric monoidal model category $\mathscr{M}$, which admits a commutative Hopf interval [BM1] Sect. 3. Then there is a transferred model structures on the category $\Alg(\mathcal{O})$ [BM2] Th.2.1. If $\mathscr{M}$ is moreover left proper model category, any map $\theta\colon \mathcal{O}_1\to \mathcal{O}_2$ of $\Sigma$-cofibrant operads in $\mathscr{M}$ produces a Quillen adjunction 
$$
L\colon \Alg(\mathcal{O}_1)\rightleftarrows \Alg(\mathcal{O}_2): R
$$
where $R=\theta^*$ and $L(-)=\mathcal{O}_2\circ_{\mathcal{O}_1}-$ [BM2] Th.4.1. This Quillen adjunction is a Quillen equivalence if $\theta$ is a weak equivalence of operads loc.cit. In particular, in the latter case the homotopy categories $\Ho(\Alg(\mathcal{O}_1)$ and $\Ho(\Alg(\mathcal{O}_2)$ are equivalent. The assumptions on $\mathscr{M}$ hold for all monoidal categories we work with, see [BM1], Examples 3.3.
\endcomment

\subsection{\sc }
The paper is organised as follows: 

\noindent In Section 2 we recall the definition and basic properties of the twisted tensor product on the category of small dg categories.

\noindent In Section 3 we recall the brace operad, and prove Lemma \ref{lemmaintro1}. Then we prove Theorem \ref{theorem1}.

\comment

\subsection*{\sc Concluding remarks} Even though we did not use the theory of $n$-operads in its full force here, there is an important $2$-operad at the background of our argument, namely  the $2$-operadic analogue of the symmetric brace operad $\Br$,
constructed by the second author in \cite{Sh2} using the twisted tensor product of dg categories. Let us denote this $2$-operad $\mathcal{B}$ (it has been denoted $\mathcal{O}$ in \cite{Sh2}).  

It was proven in \cite{BCW} that the passage from multitensors  on suitably generalised categories  enriched in a symmetric monoidal category $V$ to $2$-operads is invertible in some sense.  Any $V$-enriched $2$-operad $\mathcal{P}$ determines a multitensor (which is a slightly more general concept than a skew tensor product) on the category of $\mathcal{P}_1$-algebras, where $\mathcal{P}_1$ is the $1$-operad obtained by restriction of $\mathcal{P}$ to the 1-terminal 2-level trees (considered as $1$-ordinals).  For any algebra $X$ of $\mathcal{P}$ and any two objects $x,y\in X$ the ``Hom-space'' $X(x,y)$ has an induced $\mathcal{P}_1$-algebra structure. In the case of 2-operad $\mathcal{B}$ the operad $\mathcal{B}_1$ is  the $1$-operadic (that is, non-symmetric) version of the symmetric operad $\Assoc$, so its category of algebras is the category of dg categories.  

This observation suggests the following conceptual scheme for proving more general Swiss Cheese conjecture type statements which we secretly exploit in our paper.

Let $\mathcal{P}$ be a $V$-enriched $2$-operad. 
We can associate to $\mathcal{P}$ a Swiss Cheese type $2$-operad  $\mathcal{P}^{\mathbf{sc}}$  For this we simply introduce a coloring of the $2$-ordinals as described in [B1, Sect. 9] then we restrict $\mathcal{P}$ along the forgetful functor from colored $2$-ordinals to $2$-ordinals. It is not hard to see that the  
algebras of  $\mathcal{P}^{\mathbf{sc}}$   are pairs $(B,A)$ where $B$ is an algebra of $\mathcal{P}$ and $A$ is an algebra of $\mathcal{P}_1$ on which $B$ acts in appropriate sense.
Thus  we can  speak about the category of actions $\Act^{\mathcal{P}}(A)$.

Now assume that the multitensor associated to the $2$-operad $\mathcal{P}$  is closed from one side (meaning it has the right adjoint internal Hom, as in the case of the twisted product of dg-categories). In this case we get the following strong form of a formal Swiss Cheese conjecture, closely following the argument presented in our paper:  

\noindent {\it For a 1-terminal $\mathcal{P}_1$-algebra $A$ the set theoretical fiber $\Act^{\mathcal{P}}(A)$ is equivalent  to the comma category $\Alg(\Sym_2(\mathcal{P}))/\Hoch^\udot(A),$ where $\Sym_2(\mathcal{P})$ is the symmetrisation of $\mathcal{P}$ {\rm [B1]}.} 
Here $\Hoch^\udot(A)$ is the internal hom $[A,A](\id,\id)$.  

Therefore, if $V$ is a monoidal model category, the category $\Act(A)$ can be equipped with a model structure under some mild condition on $V$ and so  $\Ho(\Act(A))$ has a final object.   

If, moreover, $\mathcal{P}$ is contractible and cofibrant (or $\Sigma$-cofibrant in an appropriate sense), we have a Quillen equivalence $\Alg(\Sym_2(\mathcal{O})) \cong \Alg(\e_2)$ [B1], and so we would get the Swiss Cheese conjecture in its original form.

This program for understanding the Swiss Cheese conjecture (in any dimension) and its generalisations  has been presented by the first author in a lecture at University of Adelaide [B2]. 
For a partial realisation of this programme in weak 2-categorical context, see [Ju].

There is, however, an obstacle for an immediate implementing of this program in desirable generality: the question on existence of the internal Hom for a multitensor determined by a $2$-operad turns out to be quite subtle.  There is a well known counterexample of a non-closed multitensor of this kind, namely, the tensor product of Gray-categories introduced by Crans in [BCW], [C]. It seems that the best we can hope in general is that such a multitensor is only homotopically closed which fortunately would be enough  for a proof of the Swiss Cheese conjecture in its original form. We will investigate this question in our future work.   

\endcomment

\subsection*{\sc Concluding remarks} Even though we did not use the
theory of $n$-operads in its full force here, there is an important
$2$-operad at the background of our argument, namely  the $2$-operadic
analogue of the symmetric brace operad $\Br$,
constructed by the second author in \cite{Sh2} using the twisted tensor
product of dg categories. Let us denote this $2$-operad $\mathcal{B}$
(it has been denoted $\mathcal{O}$ in \cite{Sh2}).

It was proven in \cite{BCW} that the passage from multitensors  on
suitably generalised categories  enriched in a symmetric monoidal
category $V$ to $2$-operads is invertible in some sense.  Any
$V$-enriched $2$-operad $\mathcal{P}$ determines a multitensor (which is
a slightly more general concept than a skew tensor product) on the
category of $\mathcal{P}_1$-algebras, where $\mathcal{P}_1$ is the
$1$-operad obtained by restriction of $\mathcal{P}$ to the 1-terminal
2-level trees (considered as $1$-ordinals).  For any algebra $X$ of
$\mathcal{P}$ and any two objects $x,y\in X$ the ``Hom-space'' $X(x,y)$
has an induced $\mathcal{P}_1$-algebra structure. In the case of
2-operad $\mathcal{B}$ the operad $\mathcal{B}_1$ is  the $1$-operadic
(that is, non-symmetric) version of the symmetric operad $\Assoc$, so
its category of algebras is the category of dg categories. The twisted
tensor product of dg categories is exactly the multitensor associated to
the operad $\mathcal{B}$.

This observation suggests the following conceptual scheme for proving
more general Swiss-Cheese conjecture type statements which we secretly
exploit in our paper.

Let $\mathcal{P}$ be a $V$-enriched $2$-operad.
We can associate to $\mathcal{P}$ a Swiss-Cheese type $2$-operad
$\mathcal{P}^{\mathbf{sc}}$  For this we simply introduce a coloring of
the $2$-ordinals as described in [B1, Sect. 9] then we restrict
$\mathcal{P}$ along the forgetful functor from colored $2$-ordinals to
$2$-ordinals. It is not hard to see that the
algebras of  $\mathcal{P}^{\mathbf{sc}}$   are pairs $(B,A)$ where $B$
is an algebra of $\mathcal{P}$ and $A$ is an algebra of $\mathcal{P}_1$
on which $B$ acts in appropriate sense.
Thus  we can  speak about the category of actions $\Act^{\mathcal{P}}(A)$.

Now assume that the multitensor associated to the $2$-operad
$\mathcal{P}$  is closed from one side (meaning it has the right adjoint
internal Hom, as in the case of the twisted product of dg-categories). In
this case we get the following strong form of a formal Swiss-Cheese
conjecture, closely following the argument presented in our paper:

\begin{conjecture} For a 1-terminal $\mathcal{P}_1$-algebra $A$ the set
theoretical fiber $\Act^{\mathcal{P}}(A)$ is equivalent  to the comma
category $\Alg(\Sym_2(\mathcal{P}))/\Hoch^\udot(A),$ where
$\Sym_2(\mathcal{P})$ is the symmetrisation of $\mathcal{P}$ {\rm [B1]}.
where $\Hoch^\udot(A)$ is the internal hom $[A,A](\id,\id)$.
\end{conjecture}

In case of $\mathcal{P} = \mathcal{B}$ the symmetrisation
$\Sym_2(\mathcal{B})$ is exactly the operad $\Br$ and we get the
equivalence stated in Theorem \ref{theorem1}.
Notice also that the baby version of Swiss-Cheese conjecture from (1.5)
is exactly the conjecture above applied to the terminal $2$-operad in
$\mathbf{Set}$. In this case the associated multitensor is the cartesian
product of categories, the symmetrisation of this operad is the operad
$\mathbf{Com}$ and the internal hom is the usual category of functors and
natural transformations.

 For an elaborated example of our conjecture for the category of
actions of braided categories on a monoidal category  see [Ju]. In this
case the $2$-operad $\mathcal{P}$ is $\mathbf{Cat}$-enriched $2$-operad
whose algebras are Gray-categories. The corresponding multitensor is the
Gray-product of $2$-categories and the Hochschild object of a monoidal
category $A$ is its Drinfeld center with its braided monoidal structure.

If $V$ is a monoidal model category, the category $\Act(A)$ can be
equipped with a model structure under some mild condition on $V$ and so 
$\Ho(\Act(A))$ has a final object.
If, moreover, $\mathcal{P}$ is contractible and cofibrant (or
$\Sigma$-cofibrant in an appropriate sense), we have a Quillen
equivalence $\Alg(\Sym_2(\mathcal{O})) \cong \Alg(\e_2)$ [B1], and so we
would get the Swiss-Cheese conjecture in its original form.

This higher categorical understanding of the Swiss-Cheese conjecture (in
any dimension) and its generalisations  has been presented by the first
author in a lecture at University of Adelaide [B2].
There is, however, an obstacle for immediate implementation of this
program in desirable generality: the question on existence of an internal
Hom for a multitensor determined by a $2$-operad turns out to be quite
subtle.  There is a well known counterexample of a nonclosed multitensor
of this kind, namely, the tensor product of Gray-categories introduced
by Crans in \cite{BCW,C}. It seems that the best we can hope in general
is that such a multitensor is only homotopically closed which
fortunately would be enough  for the proof of the Swiss-Cheese
conjecture in its homotopical form. We will investigate this question in
our future work.

\subsubsection*{\sc Acknowledgements}

\noindent The work of B.Sh. was partially financed by the Russian Science Foundation via agreement ${\rm N\textsuperscript{\underline{o}}}$ 24-21-00119.

\section{\sc The twisted tensor product of dg categories}\label{sectiontw}
\subsection{\sc Reminder on the coherent internal Hom}\label{sectioncohhom}
Let $\k$ be a field of any characteristic, denote by $C^\udot(\k)$ the category of complexes of $\k$-vector spaces, and by $\Cat_\dg(\k)$ the category of small dg categories over $\k$. 
Let $C,D\in\Cat_\dg(\k)$. Their {\it tensor product} $C\otimes D$ is defined as the dg category with objects $\Ob(C)\times\Ob(D)$, and whose $\Hom$s are $(C\otimes D)(c_0\times d_0, c_1\times d_1)=C(c_0,c_1)\otimes_R D(d_0,d_1)$. This operation defines a symmetric monoidal structure on  $\Cat_\dg(\k)$. This symmetric monoidal structure is closed: the internal Hom $[D,E]_0$, for $D,E\in \Cat_\dg(\k)$, is defined as the dg category whose objects are dg functors $D\to E$, and for two such functors $F,G$ the complex $[C,D]_0(F,G)$ is defined as enriched natural transformations, that is, as those elements of $\prod_{d\in D}E^\udot(F(d), G(d))$ which posses the usual naturality for (complexes of) morphisms in $D$. The closeness means that
\begin{equation}\label{eq1.1}
\Hom(C\otimes D,E)\simeq \Hom(C,[D,E]_0)
\end{equation}
where $\Hom$ stands for morphisms in $\Cat_\dg(\k)$, that is, for the set of dg functors. 

The coherent Hom (denoted by $[D,E]$) has the same objects as $[D,E]_0$, but the morphisms are defined as
$$[D,E](F,G)=\Hoch^\udot(D,M_{F,G})$$
where $\Hoch^\udot(D,M)$ stands for the reduced cohomological Hochschild complex of $D$ with coefficients in a $D$-bimodule $M$, and $M_{F,G}(d_0,d_1)=E(F(d_0),G(d_1))$ is a particular (dg) $D$-bimodule. (A dg $D$-bimodule is by definition a dg functor $D^\op \otimes D\to C^\udot(\k)$).

Recall that explicitly $\Hoch^\udot(D,M)$ is the reduced totalisation of the following cosimplicial object $X^\udot(D,M)$ in $C^\udot(\k)$:
$$
X^\ell(D,M)=\prod_{d_0,\dots,d_\ell\in D}\underline{\Hom}_\k\big(D(d_{\ell-1},d_\ell)\otimes_R\dots\otimes_R D(d_0,d_1), M(d_0,d_\ell)\big)
$$
where $\underline{\Hom}_\k$ stands for the internal $\Hom$ in the category $C^\udot(\k)$. 

The cosimplicial structure on $X^\udot(D,M)$ is well-known (see e.g. [T1]). For instance, for $\Psi\in X^0(D,M)$ one has 
$$
\delta_0(\Psi)(f)=R(f)\Psi(d_1),\ \ \delta_1(\Psi)(f)=L(f)\Psi(d_0)
$$
for $f\in D(d_0,d_1)$, where $L$ and $R$ stand for the left (resp. right) action of $D$ on $M$; in our case, $\delta_0\Psi,\delta_1\Psi\in \prod_{d_0,d_1\in D}\underline{\Hom}_R(D(d_0,d_1), E(F(d_0),G(d_1))$. 

In particular, for a 0-cochain $\Psi$, one has $\Psi\in \Ker(\delta_0-\delta_1)$ iff $\Psi$ belongs posses the naturality, that is, iff $\Psi\in [D,E]_0(F,G)$. 

Note that $\Hoch^\udot(D,M)$ is the total-product-complex, the total degree is the sum of the cosimplicial degree and the degree in $C^\udot(\k)$. In particular, an element in $\Hoch^k(D,M)$ has in general infinitely many components given by elements of $\Psi_{k+a}\in X^{k+a}(D,M)$ having degree $-a$ in $C^\udot(R)$. When such an element in $\Hoch^0(D,M)$ is closed with respect the total differential $\delta+d_R$, it results in an infinite system of {\it coherent equations} on its components $\Psi_0,\Psi_1,\Psi_2,\dots$. These equations are 
$$
\delta\Psi_a=d_R\Psi_{a+1}, a=0,1,2\dots
$$
For $a=0$ and $M=M_{F,G}$ the first equation expresses that $\Psi_0(d)=E(F(d),G(d))$ is natural in $f\colon d_0\to d_1$ not on the nose but only up to a boundary:
$$
R(f)\Psi_0(d_1)-L(f)\Psi_0(d_0)=[d_R,\Psi_1(f)]
$$
and then $\Psi_1$ obeys the equation
$$
R(f)\Psi_1(f_2)-\Psi_1(f_2f_1)+L(f_2)\Psi_1(f_1)=[d_R,\Psi_2(f_2\otimes f_1)]
$$
for a chain $d_0\xrightarrow{f_1}d_1\xrightarrow{f_2}d_2$, and so on. 
It explains the name {\it coherent natural transformations} for $[D,E](F,G)$, although is can be properly applied only to closed elements (of arbitrary total degree). 

\begin{remark}\label{remcohhom}{\rm
The reason considering $[D,E]$ as a replacement of $[D,E]_0$ is that (for $D$ cofibrant) it has the correct homotopy type in the localisation with respect to quasi-equivalences of dg categories, see [LH], [F]. In general (without the assumption of cofibrancy of $D$), one has to consider the $A_\infty$ functors $D\to E$ as the objects in order to get the correct homotopy type). 
}
\end{remark}

\subsection{\sc Reminder on the twisted tensor product}\label{sectiontwtp}
It is natural to ask ourselves whether a functor $?\mapsto ?\sotimes D$ left adjoint to $?\mapsto [D,?]$ exists, so whether there is an adjunction 
\begin{equation}\label{adjj}
\Hom(C\sotimes D,E)\simeq \Hom(C,[D,E])
\end{equation}
(considered as a derived version of the classical adjunction \eqref{eq1.1}) exists. This is in fact the case, and the product $C\sotimes D$ was constructed in [Sh1] under the name {\it twisted tensor product}, and was further studied in [Sh2]. Here we briefly recall the construction of the twisted tensor product and its properties. 

Let $C$ and $D$ be two small dg categories over $\k$. Here we construct a small dg category $C\sotimes D$ over $\k$ fulfilling \eqref{adjj}. 

The set of objects of $C\sotimes D$ is $\Ob(C)\times \Ob(D)$.
Consider the graded $\k$-linear category $F(C,D)$ with objects $\Ob(C)\times \Ob(D)$ freely generated by $\{f\otimes\id_d\}_{f\in\Mor(C), d\in D}$, $\{\id_c\otimes g\}_{c\in C, g\in \Mor(D)}$, and by the new morphisms $\varepsilon(f;g_1,\dots,g_n)$, specified below. (We write below $\{C\otimes \id_d\}_{d\in D}$ assuming $\{f\otimes\id_d\}_{f\in\Mor(C), d\in D}$ etc).

Let
$c_0\xrightarrow{f}c_1$ be a morphism in $C$, and let $$d_0\xrightarrow{g_1}d_1\xrightarrow{g_2}\dots\xrightarrow{g_n}d_n$$

are chains of composable maps in $D$. For any such chains, with $n\ge 1$, one introduces a new morphism
$$
\varepsilon(f;g_1,\dots,g_n)\in \Hom((c_0,d_0),(c_1, d_n))
$$
of degree
\begin{equation}
\deg \varepsilon(f;g_1,\dots,g_n)=-n+\deg f+\sum\deg g_j
\end{equation}

The new morphisms $\varepsilon(f;g_1,\dots,g_n)$ are subject to the following identities:

\begin{itemize}
\item[$(R_1)$]
$(\id_c\otimes g_1)* (\id_c\otimes g_2)=\id_c\otimes (g_1g_2)$, $(f_1\otimes\id_d)*(f_2\otimes\id_d)=(f_1f_2)*\id_d$
\item[$(R_2)$] $\varepsilon(f;g_1,\dots,g_n)$ is linear in each argument,
\item[$(R_3)$] 
$\varepsilon(f; g_1,\dots,g_n)=0$ if $g_i=\id_y$ for some $y\in \Ob(D)$ and for some $1\le i\le n$,\\
$\varepsilon(\id_x; g_1,\dots,g_n)=0$ for $x\in\Ob(C)$ and $n\ge 1$,
\item[$(R_4)$]
for any $c_0\xrightarrow{f_1}c_1\xrightarrow{f_2}c_2$ and $d_0\xrightarrow{g_1}d_1\xrightarrow{g_2}\dots\xrightarrow{g_N}d_N$
one has:
\begin{equation}\label{eqsuper}
\varepsilon(f_2f_1;g_1,\dots,g_N)=\sum_{0\le m\le N}(-1)^{|f_1|(\sum_{j=m+1}^{N}|g_j|+N-m)}\varepsilon(f_2;g_{m+1},\dots,g_N)\star\varepsilon(f_1;g_1,\dots,g_m)
\end{equation}
\end{itemize}

To make $F(C,D)$ a dg category, one should define the differential $d\varepsilon(f;g_1,\dots,g_n)$.

For $n=1$ we set:
\begin{equation}\label{eqd0}
[d,\varepsilon (f;g)]=\pm (\id_{c_1}\otimes g)\star (f\otimes \id_{d_0})
\pm (f\otimes \id_{d_1})\star (\id_{c_0}\otimes g)
\end{equation}
(see [Sh2], 2.4.1, for the signs).

For $n\ge 2$:
\begin{equation}\label{eqd1}
\begin{aligned}
\ &[d,\varepsilon(f;g_1,\dots,g_n)]=\\
&
\pm(\id_{c_1}\otimes g_n)\star \varepsilon({f;g_1,\dots,g_{n-1}})
\pm \varepsilon({f;g_2,\dots,g_n})\star (\id_{c_0}\otimes g_1)+
\sum_{i=1}^{n-1} \pm \varepsilon({f;g_1,\dots,g_{i+1}\circ g_i,\dots,g_n})
\end{aligned}
\end{equation}
(see loc.cit. for the signs).

One proves that $d^2=0$ and that the differential agrees with relations ($R_1$)-($R_4$) above. 

Note that this construction defines a bifunctor $\Cat_\dg(\k)\times\Cat_\dg(\k)\to\Cat_\dg(\k)$, which is not symmetric in $C$ and $D$. 

It is proven in [Sh2, Prop. 2.4] that \eqref{adjj} indeed holds. 

As well,  for $C,D$ cofibrant in Tabuada model structure on $\Cat_\dg(\k)$ the natural projection $C\sotimes D\to C\otimes D$, sending all $\varepsilon(f; g_1,\dots,g_n)$, $n\ge 1$, to 0, is a quasi-equivalence of dg categories, see [Sh1, Th 2.4].

\subsection{\sc The skew-monoidal structure on $(\Cat_\dg(\k), \sotimes)$}
It is a natural question whether $\sotimes$ obeys a sort of associativity, making $(\Cat_\dg(\k),\sotimes)$ a monoidal category, possibly in a weak sense. This question has been addressed in [Sh2, Sect 3]. Here we briefly recall the results which are employed in this paper. Summing up, $(\Cat_\dg(\k), \sotimes, \underline{\k})$ is a skew-monoidal category with {\it strict} unit $\underline{\k}$ (a dg category with a single object whose endomorphisms is $\k$). Here {\it skew-monoidality} means that there is a one-sided associator
$$
\alpha\colon (C\sotimes D)\sotimes E\to C\sotimes (D\sotimes E)
$$
satisfying the usual pentagon axiom along with other standard axioms [S], [Sh2] Sect. 3.2. 
When the unit maps are not isomorphisms the coherence theorem for skew monoidal categories becomes a more complicated statement than for the classical monoidal case, see [BL]. As we have proved in [Sh2], Prop.3.5, when the unit maps are isomorphisms, the coherence is exactly as in the classical MacLane result. 
It is important for us that formula for $\alpha$ (see Theorem \ref{theoremskew}(iv) below) is almost identical to the formula for the r.h.s. of the relation \ref{eqbrace3} in the brace operad $\Br$, see Section \ref{sectionbrace}. This observation plays a key role in the proof of Lemma \ref{lemmaintro1}. 

The following statement is proven as [Sh2], Theorem 3.1:

\begin{theorem}\label{theoremskew}
For any three dg categories $C,D,E$, there is a unique dg functor
$$
\alpha_{C,D,E}\colon (C\sotimes D)\sotimes E\to C\sotimes (D\sotimes E)
$$
natural in each argument, 
which is the identity map on objects, and which is defined on morphisms as follows:
\begin{itemize}
\item[(i)]
for $f\in C$, $g\in D$, $h\in E$, $X,Y,Z$ objects of $C,D,E$ correspondingly, one has:
\begin{equation}\label{eqassoc1}
\begin{aligned}
\ &\alpha_{C,D,E}((f\star \id_Y)\star \id_Z)=f\star (\id_Y\star\id_Z)\\
&\alpha_{C,D,E}((\id_X\star g)\star\id_Z)=\id_X\star (g\star \id_Z)\\
&\alpha_{C,D,E}((\id_X\star\id_Y)\star h)=\id_X\star(\id_Y\star h)
\end{aligned}
\end{equation}
\item[(ii)]
for $f\in C$, $g_1,\dots,g_k\in D$, $k\ge 1$, and $Z$ an object in $E$, one has:
\begin{equation}\label{eqassoc1bis}
\alpha_{C,D,E}(\varepsilon(f;g_1,\dots,g_k)\star \id_Z)=\varepsilon(f;g_1\star\id_{Z},\dots, g_k\star\id_{Z})
\end{equation}

\item[(iii)]
for $f\in C$, $g\in D$, $h_1,\dots,h_n\in E$, $X$ an object of $C$, $Y$ an object of $D$, one has:
\begin{equation}\label{eqassoc1bisbis}
\begin{aligned}
\ &\alpha_{C,D,E}(\varepsilon(f\star \id_Y;h_1,\dots,h_n))=\varepsilon(f; \id_Y\star h_1,\dots,\id_Y\star h_n)\\
&\alpha_{C,D,E}(\varepsilon(\id_X\star g; h_1,\dots, h_n))=\id_X\star \varepsilon(g;h_1,\dots,h_n)
\end{aligned}
\end{equation}
\item[(iv)]
for $f\in C$, $g_1,\dots,g_k\in D$, $h_1,\dots,h_N\in E$, one has:
\begin{equation}\label{eqassoc4bis}
\begin{aligned}
\ &\alpha_{C,D,E}\big(\varepsilon(\varepsilon(f; g_1,\dots,g_k);h_1,\dots,h_N)\big)=\sum_{1\le i_1\le j_1\le i_2\le j_2\le \dots\le j_k\le N}\\
&\pm \varepsilon\big(f;\   h_1,\dots, h_{i_1}, \varepsilon(g_1;h_{i_1+1},\dots,h_{j_1}), h_{j_1+1},\dots, h_{i_2},\varepsilon(g_2;h_{i_2+1},\dots,h_{j_2}), h_{j_2+1},\dots \big)
\end{aligned}
\end{equation}
where the sum is taken over all ordered sets $\{1\le i_1\le j_1\le i_2\le j_2\le\dots\le j_{k}\le N\}$; for the case $j_\ell=i_\ell$, the corresponding term $\varepsilon(g_\ell;\ h_{i_{\ell+1}},\dots,h_{j_\ell})$ is replaced by $g_\ell\otimes \id_{-}$.
\end{itemize}

\end{theorem}
See [Sh2], Theorem 3.1 for signs in (iv). 

\qed

Next, one has

\begin{theorem}
The category $\mathscr{C}=\Cat_\dg(\k)$ of small dg categories, equipped with the twisted product $-\sotimes-$, the unit $\underline{\k}$, the associativity constrains $\alpha$, and with the natural unit isomorphisms is a skew monoidal category
\end{theorem}

See [Sh2], Theorem 3.7.

\qed

We do not reproduce here the axioms of a (closed) skew-monoidal category, referring the reader to [S] or [Sh2], Sect. 3.2.
A closed skew-monoidal category is a skew-monoidal category with internal hom, right adjoint to the tensor product. In particular, $(\Cat_\dg(\k),\sotimes,\underline{\k})$ is closed skew-monoidal. 

Note that one can define a monoid and a module over it in any skew-monoidal category, see \eqref{skewmonoid}.

One also has:
\begin{lemma}\label{lemmatwcomp}
Let $\mathscr{C}$ be a closed skew-monoidal category, with product $\otimes$ and internal hom $[-,-]$.
Then for any three objects $A,B,C\in\mathscr{C}$ one has a map
$$
M\colon [B,C]\otimes [A,B]\to [A,C]
$$
which is associative in the sense that for 4 objects $A,B,C,D\in\mathscr{C}$ the diagram
\begin{equation}
\xymatrix{
([C,D]\otimes [B,C])\otimes [A,B]\ar[rr]^{\alpha}\ar[d]_{M\otimes\id}&&[C,D]\otimes ([B,C]\otimes [A,B])\ar[d]^{\id\otimes M}\\
[B,D]\otimes [A,B]\ar[r]^{M}&[A,D]&[C,D]\otimes [A,C]\ar[l]_{M}
}
\end{equation}
and fulfils similar unit commutative diagrams [Sh2], Prop. 3.6.

In particular, for any $A\in\mathscr{C}$, $[A,A]$ is a monoid. 
\end{lemma}

\begin{proof}
We recall the construction of $M$, and refer the reader to [Sh2], Prop. 3.6 for the proof of all statements. 

Assume \eqref{adjj} holds, define $M_{A,B,C}\colon [B,C]\otimes [A,B]\to [A,C]$ as follows. First of all, there is a morphism
$$
e\colon [A,B]\otimes A\to B
$$
adjoint to $\id\colon [A,B]\to [A,B]$. One has the composition
$$
([B,C]\otimes [A,B])\otimes A\xrightarrow{\alpha}[B,C]\otimes ([A,B]\otimes A)\xrightarrow{\id\otimes e}[B,C]\otimes B\xrightarrow{e} C
$$
whose right adjoint gives $M$.
\end{proof}

\begin{remark}\label{remskewmon}{\rm
It follows from the above construction that the monoid structure on $[A,A]$ uses the associator $\alpha$. 
}
\end{remark}

\begin{lemma}\label{lemmatwfinal}
Let $\mathscr{C}$ be a closed skew-monoidal category, $B$ a monoid in $\mathscr{C}$, $A$ a module over $B$, see
\eqref{skewmonoid1}, \eqref{skewmonoid2}. Then under the ajunction in $\mathscr{C}$
$$
\mathscr{C}(B\otimes A, A)\simeq \mathscr{C}(B,[A,A])
$$
the maps defining monoid actions are in 1-to-1 correspondence with the maps of monoids.
\end{lemma}

\begin{proof}
The proof follows the same line as we discussed for the strict monoidal case in the Introduction, see \eqref{equal1}, \eqref{equal2}. We have to be more careful in the presence of (one-sided) associator in showing that the two maps in one equaliser correspond via the adjunction to the corresponding maps in another equaliser. It is straightforward, anyway, and the details are left to the reader. Note that by Remark \ref{remskewmon} (the counterpart of) \eqref{equal2} also depend on the associator. 
\end{proof}

\begin{coroll}\label{corrr}
Let $B$ be a dg algebra over $\k$ which is a $\sotimes$-monoid, $A$ a dg algebra such that there is a dg algebra map $\ell\colon B\sotimes A\to A$ such that $\ell(1_B\sotimes a)=a$ for any $a\in A$. Then $\ell$ is a monoid action if and only if the adjoint map $B\to [A,A](\id,\id)$ is a monoid map.
\end{coroll}

\qed
\comment
We will need the following generalisation of Corollary \ref{corrr} to the case of dg categories.

Let $B,A$ be dg categories, $B$ an $\sotimes$-monoid acting on $A$, $\ell\colon B\sotimes A\to A$. For each $b\in B$ denote by $\theta_b\colon A\to A$ the functor $a\mapsto \ell(\id_b\sotimes f)$, $f\in \Mor(A)$. Define $[A,A]_\ell$ as dg category whose objects are $b\in B$ and whose morphisms $[A,A]_\ell(b_1,b_2)=[A,A](\theta_{b_1},\theta_{b_2})$. Then $[A,A]_\ell\subset [A,A]$ is a full $\sotimes$-submonoid. 

\begin{coroll}\label{corrr2}
In the notations as above, the action $\ell$ gives rise to a functor $\omega_\ell\colon B\to [A,A]_\ell$ of $\sotimes$-monoids, 
and $\ell\rightsquigarrow\eta_\ell$ is a 1-to-1 correspondence between actions of $B$ on $A$ and functors of $\sotimes$-monoids $B\to [A,A]$.
\end{coroll}
\qed

\endcomment

\section{\sc The brace operad and a proof of Theorem \ref{theorem1}}\label{sectionbrace}
Recall the Getzler-Jones $B_\infty$ operad: a complex $X\in C^\udot(\k)$ is a $B_\infty$ algebra if the cofree coalgebra 
$$
T^\vee_+(X[1])=X[1]\oplus X^{\otimes 2}\oplus X[1]^{\otimes 3}\oplus\dots
$$
is a dg bialgebra (whose underlying coalgebra $T^\vee(X[1])$ is the cofree one). This definition gives rise to a symmetric dg operad called $B_\infty$. Thus to define a $B_\infty$ structure on $X$ one has to define the Taylor components of the differential
$$
p_n\colon X^{\otimes n}\to X[1-n],\  \ n\ge 1
$$
and of the product 
$$
q_{m,n}\colon X^{\otimes m}\otimes X^{\otimes n}\to X[1-m-n]
$$
which are subject to the relations which express the dg bialgebra axioms for $T^\vee_+(X[1])$. 

This dg operad is rather mysterious, in particular, its cohomology is unknown. On the other hand, its quotient by the relations $q_{m,n}=0$ for $m\ge 2$, called the {\it brace operad} $\Br$, is known to be weakly equivalent to the dg operad $C_\ldot(E_2,\k)$. In all known examples of $B_\infty$ algebras the action in fact descends to an action of $\Br$. One often also has $p_n=0$ for $n\ge 3$ (although it is not required). 

A list of examples of $\Br$-algebras includes: (1) the Hochschild cochain complex $\Hoch^\udot(A)$ of a dg algebra $A$ (or an $A_\infty$ algebra, or a small dg category) [GJ], (2) the cobar-complex of the underlying coalgebra of a dg bialgebra [Ka], [T2],  (3) 
cochain complex of 1-reduced simplicial set [Ba].

One traditionally denotes the operation $p_2(x\otimes y)=x\cup y$ and $q_{2,n}(x\otimes y_1\otimes\dots\otimes y_n)=\{x;y_1,\dots,y_n\}$. 

In this notations, the following identities are equivalent to a $\Br$-algebra structure on $V\in C^\udot(\k)$ with $p_n=0$ for $n\ge 3$:

\begin{equation}\label{eqbrace1}
x\cup(y\cup z)=(x\cup y)\cup z
\end{equation}

\begin{equation}\label{eqbrace2}
(x_1\cup x_2)\{y_1,\dots,y_n\}=\sum_{k=0}^n(-1)^{|x_2|\sum_{j=1}^k|y_j|}x_1\{y_1,\dots,y_k\}\cup x_2\{y_{k+1},\dots,y_n\}
\end{equation}

\begin{equation}\label{eqbrace3}
\begin{aligned}
\ &x\{y_1,\dots,y_m\}\{z_1,\dots,z_n\}=\sum_{0\le i_1\le j_1\le i_2\le j_2\le\dots\le i_m\le j_m \le n}\\
&(-1)^\epsilon x\{z_1,\dots,z_{i_1}, y_1\{z_{i_1+1},\dots,z_{j_1}\}, z_{j_1+1},\dots, z_{i_m}, y_m\{z_{i_m+1},\dots,z_{j_m}\}, 
z_{j_m+1},\dots,z_n\}
\end{aligned}
\end{equation}
where $\epsilon=\sum_{p=1}^m|y_p|\sum_{q=1}^{i_p}|z_q|$.

\begin{equation}\label{eqbrace4}
[d,x\{y_1,\dots,y_n\}]=\pm y_1\cup x\{y_2,\dots,y_n\}+\sum\pm x\{y_1,\dots, y_i\cup y_{i+1},\dots,y_n\}\pm x\{y_1,\dots,y_{n-1}\cup y_n
\end{equation}
with the signs as in \eqref{eqd1}.

The brace algebra with $p_n=0$ for $n\ge 3$ is called {\it unital} if there is an element $e\in V^0$ such that $e\cup x=x\cup e=x$ for any $x\in V$, and if $\{x;y_1,\dots,y_n\}=0$ when $x$ or one of $y_i$ is equal to $e$. 

\subsection{\sc A proof of Lemma \ref{lemmaintro1}}\label{prooflemma1}
Let $B$ be a dg category with a single object equipped with a dg algebra map $m\colon B\sotimes B\to B$ having the properties as in the statement of Lemma \ref{lemmaintro1}. Let $1_B$ be the unit of dg algebra $B$. By assumption, $1_B$ is also a unit for $m$. Therefore, $m(\id_B\sotimes b)=m(b\sotimes\id_B)=b$, for any $b\in B$. We have to define the product $b_1\cup b_2$ and the brace $\{b; b_1,\dots,b_n\}$. Define $b_1\cup b_2$ as the product $b_1\cdot b_2$ in dg algebra $B$, and define
$$
\{b; b_1,\dots,b_n\}:=m(\varepsilon(b; b_1,\dots,b_n))
$$
Then \eqref{eqbrace4} follows from \eqref{eqd1} and the remark on units made above, \eqref{eqbrace1} follows from the associativity of the product in $B$, \eqref{eqbrace2} follows from \eqref{eqsuper}, and finally \eqref{eqbrace3} follows from the associator formula \eqref{eqassoc4bis}. 

The unitality follows from our assumption on the units, and from $(R_3)$ in Section \ref{sectiontwtp}. 

Lemma \ref{lemmaintro1} is proven.

\qed

\subsection{\sc The operad $\mathcal{O}_0$ and comparison with [DTT]}
Define a colored operad $\mathcal{O}_0$ with two colors in $C^\udot(\k)$ requiring that $(B,A)$ is an algebra over $\mathcal{O}$ iff both $B,A$ are dg algebras over $\k$, $B$ is an $\sotimes$-monoid acting on $A$. denote, as above, the colors by $c$ and $a$. We have seen in Lemma \ref{lemmaintro1} that the restriction of $\mathcal{O}_0$ on the color $c$ is the brace operad $\Br$. One can similarly describe the entire operad $\mathcal{O}_0$. 

Namely, one can easily show that the colored dg operad $\mathcal{O}_0$ is isomorphic to the {\it normalised dg condensation} $|\sO|_\norm$ of the simplicial Swiss Cheese operad $\sO$ introduced in [DTT], Section 3. This identification is straightforward, so we leave details to the reader. (In particular, the restriction of $|\sO|_\norm$ to color $c$ is the brace operad, as is shown in detail in [BBM]).

Then Theorem 2.1 of [DTT], along with arguments similar to the ones employed in [BBM], gives that 
the colored dg operad $\mathcal{O}_0$ is weakly equivalent to the chain Swiss Cheese operad $\scc_1$.

\subsection{\sc A proof of Theorem \ref{theorem1}}
Let $B$ be $\sotimes$-monoid, and $A$ a dg algebra. By
Corollary \ref{corrr}, any twisted action of $B$ on $A$ gives a map of $\sotimes$-monoids $B\to [A,A](\id,\id)=\Hoch^\udot(A)$, and vice versa. Thus the category $\Act(\mathcal{O}_0)_A$ becomes equivalent to the comma-category of $\sotimes$-monoids over $[A,A](\id,\id)$. In particular, the action of $\Hoch^\udot(A)$ on $A$ gives the identity map of $\sotimes$-monoids $\id: [A,A](\id,\id)\to [A,A](\id,\id)$, and thus is a final object in $\Act(\mathcal{O}_0)_A$.

\qed

\bigskip

\noindent{\small
 {\sc Euler International Mathematical Institute\\
10 Pesochnaya Embankment, St. Petersburg, 197376 Russia }}

\bigskip

\noindent{{\it e-mail}: {\tt bataninmichael@gmail.com, shoikhet@pdmi.ras.ru}}

\end{document}